\documentclass{amsart}[12pt]
\usepackage{mathrsfs, amsmath,amssymb}
\usepackage{fullpage}

\newtheorem{pro}{Proposition}
\newtheorem{lem}{Lemma}
\newtheorem{cor}{Corollary}
\newtheorem{rem}{Remark}

\title{First hitting time of the boundary of a wedge of angle $\pi/4$ by a radial Dunkl process}
\author[N. Demni]{Nizar Demni}
\address{IRMAR, Universit\'e de Rennes 1\\ Campus de
Beaulieu\\ 35042 Rennes cedex\\ France}
\email{nizar.demni@univ-rennes1.fr}
\date{\today}
 
\begin{document}
\begin{abstract}
In this paper, we derive an integral representation for the density of the reciprocal of the first hitting time of the boundary of a wedge of angle $\pi/4$ by a radial Dunkl process with equal multiplicity values. Not only this representation readily yields the non negativity of the density, but also provides an analogue of Dufresne's result on the distribution of the first hitting time of zero by a Bessel process and a generalization of the Vakeroudis-Yor's identity satisfied by the first exit time from a wedge by a planar Brownian motion. We also use a result due to Spitzer on the angular part of the planar Brownian motion to prove a representation of the tail distribution of its first exit time from a dihedral wedge through the square wave function. 
\end{abstract}
\maketitle
\section{Introduction}
The radial Dunkl process is a self-similar diffusion valued in the closure of the positive Weyl chamber of a reduced root system in a finite-dimensional euclidean space (see \cite{CDGRVY}, Ch.II and III for a good account). Its infinitesimal generator is the differential part of the Dunkl-Laplace operator subject to Neumann boundary conditions. Equivalently, it is the projection of the jump Dunkl process valued in the underlying euclidian space onto the positive Weyl chamber. The importance of the radial Dunkl process in probability theory stems mainly from its connections to random matrix theory (eigenvalues of matrix-valued stochastic processes, \cite{Kat-Tan}), to the Brownian motion conditioned to stay in the interior of a Weyl chamber (representation theory of complex semi simple Lie algebras, \cite{BBO}) and to the Brownian motion reflected at the boundary of a Weyl chamber (multidimensional Tanaka formula and queuing theory, \cite{DL}). Of particular interest is also the radial Dunkl process associated with dihedral systems which has been extensively studied in \cite{Demni0}. Indeed, its angular part was identified there with a time-changed real Jacobi process by an independent Bessel clock and this identification was the key ingredient in proving several probabilistic and analytic results such as the expressions of the semi-group density and of the generalized Bessel function, a skew-product decomposition and the tail distribution of the first hitting time of the boundary of a dihedral wedge. In particular, the expansion in the basis of Jacobi polynomials obtained for the tail distribution generalizes the known Fourier expansions for the tail distribution of the exit time from dihedral wedges by a planar Brownian motion (\cite{Bra}, \cite{Com}). However, this expansion is not quite satisfactory since for instance its non negativity is far from being obvious. It does not help neither in checking whether or not interesting identities satisfied by the planar Brownian motion or for the Bessel process, such as Proposition 1.1. in \cite{Vak-Yor} or Dufresne's result on the distribution of the first hitting time of zero by a Bessel process (see e.g. \cite{MY}), extend to radial Dunkl processes valued in dihedral wedges. 

In this paper, we answer these inquiries for the radial Dunkl process valued in the dihedral wedge of angle $\pi/4$ and for sake of simplicity, we only deal with equal multiplicity values. More precisely, we derive an integral representation for the density of the reciprocal of the first hitting time of the boundary of this wedge which is easily seen to be nonnegative. Moreover, the integrand simplifies when the common multiplicity value lies in an appropriate interval and involves the modified Bessel function rather than the confluent hypergeometric function. In this case, the density is even written as the product of a Gamma weight and of its convolution with a non negative function and this result may be seen as an analogue of Dufresne's one though more complicated. On the other hand, when the radial Dunkl process starts at a point lying on the bisector, then the integrand simplifies as well and allows for a generalization of an identity due to Vakeroudis and Yor for the first exit time from the $\pi/4$-wedge by a planar Brownian motion. We also give some interest in the exit time from dihedral wedges for the latter process and derive a representation of its tail distribution through the square wave function. To this end, we use the Fourier expansion of this function proved in \cite{Bra} and valid for arbitrary (non necessarily dihedral) wedges together with the explicit expression of the characteristic function of the argument of a planar planar motion (\cite{Spi}). Let us finally stress that the choice of this dihedral wedge is justified by the difficulty of the problems solved below since for instance, we do not dispose yet of similar analytic techniques to deal even with the dihedral wedges of angles $\pi/3$ and $\pi/6$. 

The paper is organized as follows. For sake of self-containedness, we give in the next section a quick reminder on root systems and radial Dunkl processes associated with dihedral groups. We also list in the same section the various special functions occurring in the paper as well as their properties we shall need in our computations. The third section contains the statements and the proofs of the aforementioned integral representation, its simplified form when the multiplicity value lies in an appropriate interval and the generalization of Vakeroudis-Yor identity. In the last section, we show how the Fourier expansion of the tail distribution of the first exit time from dihedral wedges by a planar Brownian motion derived in \cite{Demni0} matches the one obtained in \cite{Bra} and prove afterwards its representation through the square wave function.

\section{Reminder}
\subsection{Root systems and reflection groups}
For facts on root systems, we refer the reader to the monograph \cite{Hum}. Let $(V,\langle \cdot, \cdot \rangle)$ be a finite-dimensional Euclidean space. Then, a root system $R$ in $V$ is a finite set of vectors (called roots) in $V \setminus \{0\}$ such that 
\begin{equation*}
\forall \alpha \in R, \quad \sigma_{\alpha}(R) = R, 
\end{equation*}
where $\sigma_{\alpha}$ is the reflection with respect to $\alpha^{\perp}$: 
\begin{equation*}
\forall x \in V, \quad \sigma_{\alpha}(x) := x - 2\frac{\langle \alpha, x\rangle}{\langle\alpha, \alpha \rangle}\alpha. 
\end{equation*}
In order to introduce the Dunkl operators, we assume that the root system is reduced:
\begin{equation*}
\forall \alpha \in R, \quad R \cap \mathbb{R} \alpha = \{\pm \alpha\},
\end{equation*}
but not necessarily crystallographic. Since $\{\alpha^\perp, \alpha \in R\}$ is a finite set of hyperplanes, we can choose $\beta \in V$ such that $\langle \alpha, \beta\rangle  \ne 0$ for all $\alpha \in R$. Doing so endows $R$ with a partial order and the set 
\begin{equation*}
R_+:=\{\alpha \in R:\,\langle \alpha,\beta \rangle > 0 \}
\end{equation*}
is called a positive system. To this choice of $R_+$ corresponds a unique set $S \subset R_+$  referred to as the simple system and such that every positive root $\alpha \in R_+$ is written as a positive linear combinations of simple roots.      
Besides, the set of simple reflections (those corresponding to simple roots) generates a finite group $W$, called the reflection group associated with $R$, which acts on $V$ by $x \mapsto wx$. With this action in hand, a function $k:R  \to  \mathbb{C}$ is named a multiplicity function if it is $W$-invariant:
\begin{equation*}
\forall w \in W, \quad \forall \alpha \in R, \quad k(w\alpha)=k(\alpha). 
\end{equation*}
Therefore, $k$ takes as many values as the number of orbits of the $W$-action on $R$. 

\subsection{Radial Dunkl processes associated with dihedral groups}
The dihedral group, $\mathcal{D}_2(n), n \geq 3$, is the group of orthogonal transformations that preserve a regular $n$-sided polygon in $V=\mathbb{R}^2$ centered at the origin (by convention, $\mathcal{D}_2(2) := \mathbb{Z}_2 \times \mathbb{Z}_2$). It contains $n$ rotations through multiples of $2\pi/n$ and $n$ reflections about the diagonals of the polygon. By a diagonal, we mean a line joining two opposite vertices or two midpoints of opposite sides if $n$ is even, or a vertex to the midpoint of the opposite side if $n$ is odd. Without loss of generality, one may assume that the $x$-axis is a mirror for the polygon. Then the corresponding dihedral root system, $I_2(n)$, is characterized by its positive and simple systems  given by: 
\begin{equation*}
R_+ = \{-i e^{i\pi l/n},\, 1 \leq l \leq n\},\quad S = \{e^{i\pi/n}e^{-i\pi/2}, e^{i\pi/2}\},
\end{equation*}  
so that the positive Weyl chamber $C$ is the wedge of angle $\pi/n$:
\begin{equation*}
C= \{(r, \theta), \, r > 0, \, 0 < \theta < \pi/n\}.  
\end{equation*}
In particular, the crystallographic systems are $I_2(3)$ which is isomorphic to $R=A_2$, $I_2(4) = B_2$ and $I_2(6)$ which is isomorphic to the exceptional root system of type $G_2$ (\cite{Dun-Xu}). When $n = 2p, p\geq 2$, the roots form two orbits so that there are two multiplicity values, say $k_0,k_1$. Otherwise, there is only one orbit and we denote by $k$ the corresponding multiplicity value. If the multiplicity function is non negative, the radial Dunkl process associated with an even dihedral group 
$\mathcal{D}_2(2p), p \geq 1$, is the diffusion $X$ valued in the dihedral wedge $\overline{C}$ whose infinitesimal generator acts on smooth functions as\footnote{The radial Dunkl process associated with odd dihedral groups $\mathcal{D}_2(n), n \geq 3$ is defined similarly via the subsitutions: $k_1 \rightarrow 0, k_0 \rightarrow k, p \rightarrow n$.}:
\begin{equation*}
\frac{1}{2}\left[\partial_r^2  + \frac{2\gamma +1}{r}\partial_r\right]  + \frac{1}{r^2}\left[\frac{\partial_{\theta}^2}{2} +p(k_0\cot(p\theta) - k_1\tan(p\theta))\partial_{\theta}\right],
\end{equation*}
subject to Neumann boundary conditions, where $\gamma := p(k_0+k_1)$ (\cite{Demni0}).  Let  
\begin{equation*}
T_0 := \inf\{t, X_t \in \partial C\},
\end{equation*}
be the first hitting time of $\partial C$ and recall from \cite{CDGRVY}, CH.II, that $T_0$ is almost surely finite when at least one multiplicity value is strictly less than $1/2$. Equivalently, define the so-called indices:
\begin{equation*}
\nu_j := k_j-(1/2), \,\, j \in \{0,1\},
\end{equation*}
and assume $1/2 \leq k_j \leq 1, j \in \{0,1\},$ with $k_0 > 1/2$ and/or $k_1 > 1/2$. Then $-\nu_j = (1-k_j) - 1/2$ are the indices corresponding to the multiplicity values $1-k_j, j \in \{0,1\},$ and the underlying radial Dunkl process hits almost surely $\partial C$ since  $1-k_1 < 1/2$ and/or $1-k_0 < 1/2$. In this respect, we shall denote $\mathbb{P}_{\rho,\phi}^{(-\nu_1, -\nu_0)}$ the probability law of a radial Dunkl process with indices $(-\nu_1, -\nu_0)$ and starting at $x = \rho e^{i\phi} \in C$.

\subsection{Special functions}
In this paragraph, we record the definitions of various special functions occurring in the remainder of the paper as well as some of their properties we will need in our subsequent computations. The reader is referred for instance to the standard books 
\cite{AAR}, \cite{Erd}. 
We start with the Gamma function: 
\begin{equation*}
\Gamma(x) = \int_0^{\infty} e^{-u}u^{x-1} du, \quad x > 0, 
\end{equation*}
and the Legendre duplication formula: 
\begin{equation}\label{Leg}
\sqrt{\pi}\Gamma(2x+1) = 2^{2x-1}\Gamma\left(x+\frac{1}{2}\right)\Gamma(x+1).
\end{equation}
Then, we recall the Pochhammer symbol:
\begin{equation*}
(a)_k = (a+k-1)\dots(a+1)a, \quad a \in \mathbb{R}, \, k \in \mathbb{N}, 
\end{equation*}
with the convention $(0)_k = \delta_{k0}$. If $a > 0$ then
\begin{equation*}
(a)_k = \frac{\Gamma(a+k)}{\Gamma(a)}
\end{equation*} 
while
\begin{equation}\label{I1}
\frac{(-n)_k}{k!} = (-1)^k \binom{n}{k}
\end{equation}
if $k \leq n$ and $(-n)_k = 0$ otherwise. Next comes the generalized hypergeometric function defined by the series 
\begin{equation*}
{}_rF_q((a_i, 1 \leq i \leq r), (b_j, 1 \leq j \leq q); z) = \sum_{m \geq 0}\frac{\prod_{i=1}^r(a_i)_m}{\prod_{j=1}^q(b_j)_m}\frac{z^m}{m!},
\end{equation*}
provided it converges absolutely. Here, an empty product equals one and the parameters $(a_i, 1 \leq i \leq r)$ are reals while $(b_j, 1 \leq j \leq q) \in \mathbb{R} \setminus -\mathbb{N}$. With regard to \eqref{I1}, 
the hypergeometric series terminates when at least $a_i = -n \in - \mathbb{N}$ for some $1 \leq i \leq r$, therefore reduces in this case to a polynomial of degree $n$. For instance, the family $(P_j^{(a,b)})_{j \geq 0}$ of Jacobi polynomials is defined by 
\begin{equation*}
P_j^{(a,b)}(u) = \frac{(a+1)_j}{j!} {}_2F_1(-j, j+a+b+1, a+1; u), \quad a,b > -1.
\end{equation*}
Here, ${}_2F_1$ is the Gauss hypergeometric function which admits the Euler integral representation: 
\begin{equation}\label{Gauss}
{}_2F_1(c,d,e; u) = \frac{\Gamma(e)}{\Gamma(c-d)\Gamma(d)} \int_0^1 (1-uz)^{-c}z^{d-1}(1-z)^{e-d-1} dz, \quad e > d > 0, |u| < 1, c \in \mathbb{R},
\end{equation}
and satisfies the quadratic transformation: 
\begin{equation}\label{Quad}
{}_2F_1(a,b,2a;u) = \frac{1}{(1-(u/2))^b}{}_2F_1\left(\frac{b}{2}, \frac{b+1}{2}, a+\frac{1}{2}; \frac{u^2}{(2-u)^2}\right), \quad |\arg(1-u)| < \pi.  
\end{equation}
The squared $L^2$-norm of the $j$-th Jacobi polynomial with respect to the weight $(1-u)^a(1+u)^b$ is given by 
\begin{equation}\label{SN}
||P_j^{(a,b)}||_2 = \frac{2^{a+b+1}}{2j+a+b+1} \frac{\Gamma(j+a+1)\Gamma(j+b+1)}{\Gamma(j+a+b+1) j!},
\end{equation}
and we shall denote $p_j^{(a,b)}$ the $j$-th orthonormal Jacobi polynomial $p_j^{(a,b)} := P_j^{(a,b)}/||P_j^{(a,b)}||_2$. Moreover, we shall need the following values: 
\begin{equation}\label{SpeVal}
P_j^{(a,b)}(1) = \frac{(a+1)_j}{j!}, \quad P_j^{(a,b)}(-1) = (-1)^j\frac{(b+1)_j}{j!},
\end{equation}
as well as the differentiation rule:
 \begin{equation}\label{Differ}
\frac{d}{du} P_{j+1}^{(a-1,b-1)}(u) = \frac{j+a+b}{2}P_{j}^{(a,b)}(u).
\end{equation}
When $a=b = \nu-1/2$, $P_j^{(a,a)}$ reduces up to a normalization to the $j$-th Gegenbauer polynomial of parameter $\nu$:
\begin{equation}\label{GegJac}
C_j^{(\nu)}(z) = \frac{(2\nu)_j}{(\nu+1/2)_j}P_j^{(\nu-1/2,\nu-1/2)}(z)  = \frac{(2\nu)_j}{j!}{}_2F_1\left(-j, j+2\nu, \nu+\frac{1}{2}, \frac{1-z}{2}\right).
\end{equation}
Now, the modified Bessel function of index $\kappa \in \mathbb{R}$ is defined by:  
\begin{equation*}
I_{\kappa}(u) := \sum_{j \geq 0}\frac{1}{\Gamma(\kappa+j+1) j!} \left(\frac{u}{2}\right)^{2j+\kappa}, \quad u \in \mathbb{C}.
\end{equation*}
We shall also denote by
\begin{equation*}
i_{\kappa}(u) := \left(\frac{2}{u}\right)^{\kappa}\Gamma(\kappa+1)I_{\kappa}(u) 
\end{equation*}
the normalized modified Bessel function and
\begin{equation*}
\mu^{s}(du) = \frac{\Gamma(s+1/2)}{\sqrt{\pi}\Gamma(s)}(1-u^2)^{s-1} {\bf 1}_{[-1,1]}(u)du, \quad s > 0,
\end{equation*}
the symmetric Beta distribution. Thus, the Poisson integral representation of $i_{\kappa}$ takes the form:
\begin{equation}\label{Poisson}
i_{\kappa -1/2}(u) = \int e^{zu} \mu^{\kappa}(du), \quad \kappa > -1/2. 
\end{equation}
Finally, the confluent hypergeometric function ${}_1F_1$ admits the following integral representation: 
\begin{align}
{}_1F_1\left(a, b, v\right) 
& = \frac{\Gamma(b)}{\Gamma(a)\Gamma(b-a)}\int_0^1e^{uv}u^{a-1}(1-u)^{b-a-1}du, \quad b > a > 0, \label{Euler}
\end{align}
and satisfies the Kummer relations:
\begin{eqnarray}
e^{-z}{}_1F_1(a,b;z) & = &{}_1F_1(b-a,b;-z), \label{Kum1} \\ 
{}_1F_1(a ,2a+1; x) & = & 2^{2a-1}\Gamma\left(a+\frac{1}{2}\right)e^{x/2}(-x)^{(1/2)-a}\left[{\it I}_{a -(1/2)}\left(-\frac{x}{2}\right) + {\it I}_{a +(1/2)}\left(-\frac{x}{2}\right)\right] \label{Kum2}.
\end{eqnarray}

\section{First hitting time of the boundary of a wedge of angle $\pi/4$ by a radial Dunkl process} 
Let $x = \rho e^{i\phi}, \rho > 0, 0 < \phi < \pi/4$ belong to the positive Weyl chamber associated with the even dihedral group $\mathcal{D}_2(4)$. In this section, we shall derive an integral representation for the density of 
\begin{equation*}
V_0 := \frac{\rho^2}{2T_0}
\end{equation*}
under the probability law $\mathbb{P}_{\rho,\phi}^{(-\nu, -\nu)}$, where we assume for sake of simplicity that $k_0 = k_1 := k \in (1/2,1]$ and set $\nu:= k-1/2$. To this end, recall from \cite{Demni0}, Proposition 3, that if $1/2 \leq k_0, k_1 \leq 1$ with at least one value strictly larger $1/2$, then \footnote{There is a misprint in the statement of Proposition 3 in \cite{Demni0}: $(\rho/\sqrt{t})^{\gamma - 2p}$ should be $(\rho/\sqrt{t})^{2\gamma - 2p}$.}:
\begin{multline}\label{Tail}
\mathbb{P}_{\rho,\phi}^{(-\nu_1, -\nu_0)}(T_0 > t) = c_{p,\nu_0,\nu_1} \sin^{2\nu_0}(p\phi) \cos^{2\nu_1}(p\phi) e^{-\rho^2/(2t)} 
\\ \sum_{j \geq 0} F(j) \left(\frac{\rho^2}{2t}\right)^{p(j+\nu_0+\nu_1)}{}_1F_1\left(a_j + 1, b_j+1, \frac{\rho^2}{2t}\right)p_j^{(\nu_1, \nu_0)}(\cos(2p\phi)),
\end{multline}
where $c_{p,\nu_0,\nu_1}$ is a normalizing constant, 
\begin{equation*}
a_j := p(j+1), \quad  b_j := p(2j+k_0+k_1) = p(2j+\nu_0+\nu_1+1),
\end{equation*}
and
\begin{equation*}
F(j) := \frac{\Gamma(a_j+1)}{\Gamma(b_j+1)} \int_{-1}^{1}p_j^{(\nu_1,\nu_0)}(s)ds = \frac{\Gamma(a_j+1)}{||P_j^{(\nu_1,\nu_0)}||_2\Gamma(b_j+1)} \int_{-1}^{1}P_j^{(\nu_1,\nu_0)}(s)ds.
\end{equation*}
Then, the issue of our computations is: 
\begin{pro}
Under $\mathbb{P}_{\rho,\phi}^{(-\nu, -\nu)}$, the density of $V_0$ is given up to a normalizing constant by: 
\begin{multline*}
\sin^{2\nu}(4\phi)e^{-v}v^{4\nu-1} \int \left\{{}_1F_1\left(2, 2\nu+\frac{3}{2}; \frac{v(1-\cos(2\phi)u)}{2}\right) + \right. 
\\ \left. {}_1F_1\left(2, 2\nu+\frac{3}{2}; \frac{v(1-\sin(2\phi)u)}{2}\right)\right\} \mu^{\nu+(1/2)}(du).
\end{multline*}
\end{pro}
\begin{proof}
Since the proof is lengthy and quite technical, we shall omit the various normalizing constants which possibly and only depend on $p,k$ and simply write $\propto$ to mean 'proportional to'. Moreover, we firstly work out the right-hand side of \eqref{Tail} and prove secondly an auxiliary result which will be also used for later purposes. Using \eqref{SN}, \eqref{SpeVal} and \eqref{Differ}, we readily get 
\begin{align*}
 \frac{F(j)}{||P_j^{(\nu_1,\nu_0)}||_2} &= \frac{2}{(j+\nu_0+\nu_1)||\left(P_j^{(\nu_1,\nu_0)}||_2\right)^2}\frac{\Gamma(a_j+1)}{\Gamma(b_j+1)}  \left[P_{j+1}^{(\nu_1-1,\nu_0-1)}(1) - P_{j+1}^{(\nu_1-1,\nu_0-1)}(-1)\right]
 \\ &\propto \frac{\Gamma(j+\nu_0+\nu_1)\Gamma(a_j)}{\Gamma(b_j)}
 \left[\frac{1}{\Gamma(\nu_1)\Gamma(j+\nu_0+1)} + \frac{(-1)^j}{\Gamma(\nu_0)\Gamma(j+\nu_1+1)}\right].
 \end{align*}
Under the assumption $k_0=k_1 = k$, $F(2j+1) = 0$ while 
\begin{align*}
F(2j) &\propto \frac{\Gamma(a_{2j})}{\Gamma(b_{2j})}\frac{\Gamma(2(j+\nu))}{\Gamma(2j+\nu+1)}, \,\, j\geq 0.
\end{align*}
As a result, 
\begin{multline*}
\mathbb{P}_{\rho,\phi}^{(-\nu, -\nu)}(T_0 > t) \propto    \sin^{2\nu}(2p\phi) e^{-\rho^2/(2t)}\sum_{j \geq 0} \frac{\Gamma(a_{2j})}{\Gamma(b_{2j})}\frac{\Gamma(2(j+\nu))}{\Gamma(2j+\nu+1)}
\\ \left(\frac{\rho^2}{2t}\right)^{2p(j+\nu)}{}_1F_1\left(a_{2j}+1, b_{2j}+1, \frac{\rho^2}{2t}\right)P_{2j}^{(\nu, \nu)}(\cos(2p\phi)).
\end{multline*}
Using \eqref{GegJac}, the last expression may be written as
\begin{multline*}
\mathbb{P}_{\rho,\phi}^{(-\nu,-\nu)} (T_0 > t) \propto   \sin^{2\nu}(2p\phi) e^{-\rho^2/(2t)}  \\ \sum_{j \geq 0} \frac{\Gamma(a_{2j})}{(2j+2\nu)\Gamma(b_{2j})} \left(\frac{\rho^2}{2t}\right)^{2p(j+\nu)}{}_1F_1\left(a_{2j}+1, b_{2j}+1, \frac{\rho^2}{2t}\right)
C_{2j}^{(\nu+1/2)}(\cos(2p\phi)),
\end{multline*}
or equivalently as
\begin{align*}
\mathbb{P}_{\rho,\phi}^{(-\nu,-\nu)}(V_0 < v) \propto  & \sin^{2\nu}(2p\phi) e^{-v}\sum_{j \geq 0} \frac{\Gamma(a_{2j})}{(2j+2\nu)\Gamma(b_{2j})} v^{2p(j+\nu)}{}_1F_1\left(a_{2j}+1, b_{2j}+1, v\right) C_{2j}^{(\nu+1/2)}(\cos(2p\phi)),
\end{align*}
where $v := \rho^2/(2t)$. This series is the even part of 
\begin{align*}
\sin^{2\nu}(2p\phi) e^{-v}\sum_{j \geq 0} \frac{\Gamma(a_j)}{(j+2\nu)\Gamma(b_j)} v^{p(j+2\nu)}{}_1F_1\left(a_j+1, b_j+1, v\right)  C_{j}^{(\nu+1/2)}(\cos(2p\phi))
\end{align*}
which, by Kummer first relation \eqref{Kum1}, transforms into 
\begin{align*}
\sin^{2\nu}(2p\phi) \sum_{j \geq 0} \frac{\Gamma(a_j)}{(j+2\nu)\Gamma(b_j)} v^{p(j+2\nu)}{}_1F_1\left(b_j-a_j, b_j+1, -v\right)  C_{j}^{(\nu+1/2)}(\cos(2p\phi)).
\end{align*}
Since $b_j - a_j = 2p(j+k) - p(j+1) = p(j+2\nu)$, then 
\begin{equation*}
\frac{v^{p(j+2\nu)}}{p(j+2\nu)}{}_1F_1\left(b_j-a_j, b_{j}+1, -v\right) = \sum_{m \geq 0}(-1)^m \frac{(p(j+2\nu)+1)_m}{(b_j+1)_m} \frac{v^{m+p(j+2\nu)}}{(p(j+2\nu) + m)m!}
\end{equation*}
and the derivative of the right-hand side of this last equality with respect to $v$ is given by 
\begin{align*}
\sum_{m \geq 0}(-1)^m \frac{(p(j+2\nu)+1)_m}{(b_j+1)_m} \frac{v^{m+p(j+2\nu)-1}}{m!} & = v^{p(j+2\nu)-1}{}_1F_1\left(b_j-a_j+1, b_j+1, -v\right)
\\& = e^{-v}v^{p(j+2\nu)-1}{}_1F_1\left(a_j, b_j+1, v\right).
\end{align*}
As a result, the density of $V_0$ is proportional to the even part of the series 
\begin{equation}\label{Serie1}
\sin^{2\nu}(2p\phi) e^{-v}\sum_{j \geq 0}\frac{\Gamma(a_j)}{\Gamma(b_j)}v^{p(j+2\nu)-1}{}_1F_1\left(a_j, b_j+1, v\right)C_{j}^{(k)}(\cos(2p\phi)).
\end{equation}
When $p=2$, $a_j = 2(j+1), b_j = 4(j+k)$ and this series admits the following integral representation:
\begin{lem}\label{lem1}
For any $k > 0$, 
\begin{multline*}
\sum_{j \geq 0}\frac{\Gamma(2(j+1))}{\Gamma(4(j+k))}v^{2j}{}_1F_1\left(2(j+1), 4(j+k) + 1, v\right)C_{j}^{(k)}(\cos(4\phi)) = \frac{1}{\Gamma(4k)} \\ \int {}_1F_1\left(2, 2k+\frac{1}{2}; \frac{v(1-\cos(2\phi)u)}{2}\right)\mu^{k}(du).
\end{multline*}
\end{lem}
Assume for a while that the lemma holds true. Then, the symmetry relation
\begin{equation}\label{Symmetry}
(-1)^jC_{j}^{(k)}(\cos(2p\phi)) = C_j^{(k)}(-\cos(2p\phi)) = C_j^{(k)}\left(\cos\left(2p\left(\frac{\pi}{2p} - \phi\right)\right)\right)
\end{equation}
shows that 
\begin{multline*}
\sum_{j \geq 0}(-1)^j\frac{\Gamma(2(j+1))}{\Gamma(4(j+k))}v^{2j}{}_1F_1\left(2(j+1), 4(j+k) + 1, v\right)C_{j}^{(k)}(\cos(4\phi)) = \frac{1}{\Gamma(4k)} \\ \int {}_1F_1\left(2, 2k+\frac{1}{2}; \frac{v(1-\sin(2\phi)u)}{2}\right)\mu^{k}(du).
\end{multline*}
Recalling $\nu = k-1/2$, the proposition is proved. 
\end{proof}

\begin{proof}[Proof of the lemma]
It mainly follows from an identity satisfied by Gegenbauer polynomials and from Erdelyi multiplication Theorem. Actually, the following identity was noticed in \cite{Xu}: for any $x \in [-1,1]$,
\begin{equation}\label{IR1}
C_j^{(k)}(2x^2 -1) = \int_{-1}^1 C_{2j}^{(2k)}(ux) \mu^{k}(du),
\end{equation}
while Erdelyi multiplication Theorem asserts that (\cite{Er}, p.385, (7)):
\begin{equation}\label{EMT}
{}_1F_1(a,c;yz)  = \sum_{j \geq 0} \frac{(-z)^j\Gamma(b+j)(a)_j}{\Gamma(b+2j)j!} {}_2F_1(-j, j+b, c; y){}_1F_1(a+j, b+1+2j; z),
\end{equation}
where $z,a,b,c > 0, |y| < 1$. Now, \eqref{IR1} leads to 
\begin{equation*}
 \int \sum_{j \geq 0}\frac{\Gamma(2(j+1))}{\Gamma(4(j+k))}v^{2j}{}_1F_1\left(2(j+1), 4(j+k) + 1, v\right)C_{2j}^{(2k)}(u\cos(2\phi))\mu^{k}(du) 
\end{equation*} 
which is the sum of 
\begin{equation*}
\frac{1}{2} \int \left\{\sum_{j \geq 0}(\pm 1)^j \frac{\Gamma(j+2)}{\Gamma(2j+4k)}v^{j}{}_1F_1\left(j+12, 2j+4k + 1, v\right)C_{j}^{(2k)}(u\cos(2\phi))\right\}\mu^{k}(du). 
\end{equation*} 
Since $\mu^k$ is symmetric, then \eqref{Symmetry} shows that both series give the same contribution after integration with respect to $u$. Besides, the second equality of \eqref{GegJac} shows that 
\begin{multline*}
\sum_{j \geq 0}\frac{\Gamma(j+2)}{\Gamma(2j+4k)}(-v)^{j}{}_1F_1\left(j+12, 2j+4k + 1, v\right)C_{j}^{(2k)}(u\cos(2\phi)) = \sum_{j \geq 0}\frac{(2)_j(4k)_j}{\Gamma(2j+4k)j!} (-v)^{j}
\\ {}_1F_1\left(j+2, 2j+4k + 1, v\right){}_2F_1\left(-j, j+4k, 2k + \frac{1}{2}; \frac{1-u\cos(2\phi)}{2}\right)
\end{multline*}
which, owing to Erdelyi multiplication Theorem specialized to $a = 2, b=4k, c = 2k+(1/2), z=v, y = (1-u\cos(2\phi))/2$, proves the lemma.
\end{proof}

\begin{cor}
If $\nu \in (1/4, 1/2]$, then the density of $V_0$ reads (up to a normalizing constant):
\begin{equation*}
\sin^{2\nu}(4\phi)e^{-v/2}v^{2\nu-(3/2)} \int_0^ve^{-y/2}y^{2\nu-(3/2)}(v-y)\left[i_{\nu}\left(\frac{\cos(2\phi)}{2}(v-y)\right) + i_{\nu}\left(\frac{\sin(2\phi)}{2}(v-y)y\right)\right] dy.
\end{equation*}
 \end{cor}
\begin{proof}
The assumption made on $\nu$ ensures the validity of the integral representation \eqref{Euler}: 
\begin{equation*}
{}_1F_1\left(2, 2\nu+\frac{3}{2}; \frac{v(1-u\cos(2\phi))}{2}\right) \propto \int_0^1e^{yv(1-u\cos(2\phi)/2}y(1-y)^{2\nu-(3/2)} dy.
\end{equation*}
Using Fubini Theorem and \eqref{Poisson}, it follows that 
\begin{equation*}
\int {}_1F_1\left(2, 2\nu+\frac{3}{2}; \frac{v(1-\cos(2\phi)u)}{2}\right)\mu^{\nu+(1/2)}(du) = \int_0^1e^{vy/2}i_{\nu}\left(\frac{v\cos(2\phi)}{2}y\right)y(1-y)^{2\nu-(3/2)}dy
\end{equation*}
and similarly 
\begin{equation*}
\int {}_1F_1\left(2, 2\nu+\frac{3}{2}; \frac{v(1-\sin(2\phi)u)}{2}\right)\mu^{\nu+(1/2)}(du) = \int_0^1e^{vy/2}i_{\nu}\left(\frac{v\sin(2\phi)}{2}y\right)y(1-y)^{2\nu-(3/2)}dy.
\end{equation*}
Together with the variable change $y \mapsto y/v$ yield the formula:
\begin{equation*}
\sin^{2\nu}(4\phi)e^{-v}v^{2\nu-(3/2)} \int_0^ve^{y/2}\left[i_{\nu}\left(\frac{\cos(2\phi)}{2}y\right) + i_{\nu}\left(\frac{\sin(2\phi)}{2}y\right)\right] y(v-y)^{2\nu-(3/2)} dy.
\end{equation*}
Multiplying the integral by $e^{-v/2}$ and performing the variable change $y \mapsto v-y$, the corollary is proved 
\end{proof}

\begin{rem}
In the rank-one case $R = \{\pm 1\}$, the reflection group is $\mathbb{Z}_2$ and the positive Weyl chamber is $(0,\infty)$. In this case, the radial Dunkl process is a Bessel process of index $\nu = k-(1/2)$ and it is already known that this process hits zero almost surely when $k \in [0,1/2)$ (\cite{Rev-Yor}, CH.XI). Besides, the distribution of the first hitting time of zero is the reciprocal of a Gamma distribution of parameter $\nu$ and this fact follows for instance from a result due to Dufresne on the distribution of the exponential functional of a Brownian motion of drift $\nu$ (see e.g. \cite{MY}, p.3). As a matter of fact, for the abelian group $\mathcal{D}_2(2) = \mathbb{Z}_2 \times \mathbb{Z}_2$, $T_0$ is the infimum of two independent reciprocal Gamma variables of parameter $\nu$ so that the density of $V_0$ reads 
\begin{equation*}
(\sin\phi)^{2\nu}v^{\nu-1}e^{-v\sin^2\phi} \int_0^{v\cos^2\phi} e^{-u} u^{\nu-1}du + (\cos\phi)^{2\nu}v^{\nu-1}e^{-v\cos^2\phi} \int_0^{v\sin^2\phi} e^{-u} u^{\nu-1}du.
\end{equation*}
 Hence, the previous corollary provides an analogue of Dufresne's result for the non abelian group $\mathcal{D}_2(4)$. 
\end{rem}

\begin{cor}
If $\phi = \pi/8$ then, for any $y \geq 0$, 
\begin{align*}
\mathbb{E}_{\rho, \pi/8}^{(-\nu,-\nu)}\left(V_0^{(3/2)-2\nu}e^{-yV_0}\right)  \propto \frac{1}{(1+y)^{2\nu- (1/2)}(1+2y)^2} {}_2F_1\left(1, \frac{3}{2}, \nu+ 1; \frac{1}{2(1+2y)^2}\right).
\end{align*}
\end{cor}
\begin{proof}
When $\phi = \pi/8$, the expression of the density of $V_0$ simplifies to (up to a constant): 
\begin{equation*}
e^{-v}v^{4\nu-1} \int {}_1F_1\left(2, 2\nu+\frac{3}{2}; \frac{v(1-u/\sqrt{2})}{2}\right)\mu^{\nu+(1/2)}(du).
\end{equation*}
Now, multiply by $v^{-2\nu + (3/2)}e^{-yv}, y > 0$, expand the confluent hypergeometric function and integrate with respect to $v$ to get 
\begin{align*}
\mathbb{E}_{\rho, \pi/8}^{(-\nu,-\nu)}\left(V_0^{(3/2)-2\nu}e^{-yV_0}\right) & \propto \frac{1}{(1+y)^{2\nu+(3/2)}}\int {}_1F_0\left(2, \frac{(1-u/\sqrt{2})}{2(1+y)}\right)\mu^{\nu+(1/2)}(du)
\\& \propto \frac{1}{(1+y)^{2\nu+(3/2)}}\int_{-1}^1 \frac{(1+y)^2}{(1+2y - u/\sqrt{2})^2}(1-u^2)^{\nu-(1/2)}du
\\& \propto \frac{1}{(1+y)^{2\nu- (1/2)}[1-\sqrt{2}(1+2y)]^2} {}_2F_1\left(2,\nu+\frac{1}{2}, 2\nu+1; \frac{2}{1-\sqrt{2}(1+2y)}\right)
\end{align*}
where the last line follows from the variable change $z=(1-u)/2$ and \eqref{Gauss}. Applying the quadratic transformation \eqref{Quad}, the proposition follows after few computations. 
\end{proof}

\begin{rem}
The index value $-\nu=-1/2$ corresponds to a null multiplicity function in which case the radial Dunkl process reduces to a planar Brownian motion reflected at $\partial C$. Since the reflected and the non reflected processes have the same distribution up to $T_0$, then $T_0$ has the same distribution in this case as the first exit time from $C$ by a planar Brownian motion. In particular, the result of the corollary coincides with Proposition 1.1. in \cite{Vak-Yor} for the special value $c = \pi/8$: 
\begin{equation*}
\mathbb{E}_{\rho, \pi/8}^{(-1/2,-1/2)}\left(V_0^{1/2}e^{-yV_0}\right) \propto \frac{1}{\sqrt{1+y}[2(1+2y)^2-1]}.
\end{equation*}
\end{rem}

\section{Exit time of a planar Brownian motion from a dihedral wedge}
In this section, we give a particular interest to the exit time of a planar Brownian motion from a dihedral wedge which already attracted the attention of probabilists and physicists. For instance, the tail distribution of this random variable was derived in \cite{Dou-Oco} using the reflection principle and combinatorial arguments, and in an unpublished manuscript by A. Comtet who solved the heat equation in an arbitrary wedge with Dirichlet boundary conditions (\cite{Com}). Another expression of this probability was also obtained in \cite{Bra} for arbitrary wedges as a series of modified Bessel functions. For even dihedral wedges $\pi/(2p), p \geq 1$, the latter can be matched with \eqref{Tail} specialized to $k_0 = k_1 = 1$ as follows (a similar claim holds true for odd dihedral wedges). Using \eqref{Kum1}, \eqref{Tail} is written in this case as:
\begin{multline*}
\mathbb{P}_{\rho,\phi}^{(-1/2,-1/2)}(T_0 > t) \propto \sin(2p\phi) e^{-\rho^2/(2t)}
\\ \sum_{j \geq 0} F(j) \left(\frac{\rho^2}{2t}\right)^{p(j+1)}{}_1F_1\left((j+1)p+1, 2(j+1)p+1, \frac{\rho^2}{2t}\right)p_j^{(1/2,1/2)}(\cos(2p\phi))
\\ = \sin(2p\phi)  \sum_{j \geq 0} F(j) \left(\frac{\rho^2}{2t}\right)^{p(j+1)}{}_1F_1\left((j+1)p, 2(j+1)p+1, -\frac{\rho^2}{2t}\right)p_j^{(1/2,1/2)}(\cos(2p\phi)),
\end{multline*}
where we recall
\begin{equation*}
F(j) = \frac{\Gamma(p(j+1)+1)}{\Gamma(2p(j+1)+1)} \int_{-1}^{1}p_j^{(1/2,1/2)}(s)ds = \frac{\Gamma(p(j+1)+1)}{||P_j^{(1/2,1/2)}||_2\Gamma(2p(j+1)+1)} \int_{-1}^{1}P_j^{(1/2,1/2)}(s)ds.
\end{equation*}
Now, let $U_j$ be the $j$-th Tchebycheff polynomial of the second kind: 
\begin{equation*}
U_j(\cos \phi) := \frac{\sin[(j+1)\phi]}{\sin\phi} = \frac{(j+1)!}{(3/2)_j} P_j^{(1/2,1/2)}(\cos\phi),
\end{equation*}
and use \eqref{Leg} and \eqref{SN} to derive
\begin{eqnarray*}
\Gamma(2p(j+1)+1) &=& \frac{2^{2p(j+1)}}{\sqrt{\pi}} \Gamma\left(p(j+1) + \frac{1}{2}\right)\Gamma(p(j+1)+1) \\ 
||P_j^{(1/2,1/2)}||_2 &=& 2 \left[\frac{\Gamma(j+(3/2))}{(j+1)!}\right]^2.
\end{eqnarray*}
Then, \eqref{Kum2} entails
\begin{equation*}
\mathbb{P}_{\rho,\phi}^{(-1/2,1/2)}(T_0 > t) \propto e^{-\rho^2/(4t)} \sqrt{\frac{\rho^2}{2t}}\sum_{j \geq 0} S(j) \left[{\it I}_{(j+1)p -(1/2)}\left(\frac{\rho^2}{4t}\right) + {\it I}_{(j+1)p +(1/2)}\left(\frac{\rho^2}{4t}\right)\right] \sin[2(j+1)p\phi],
\end{equation*}
where
\begin{equation*}
S(j) := \int_{0}^{\pi} \sin((j+1)y) dy = \frac{1-\cos((j+1)\pi)}{(j+1)}, \quad j \geq 0.
\end{equation*}
Since $S(j) = 0$ when $j$ is odd and $S(j) = 2/(j+1)$ when $j$ is even, then 
\begin{equation*}
\mathbb{P}_{\rho,\phi}^{(-1/2,1/2)}(T_0 > t) \propto e^{-\rho^2/(4t)} \sqrt{\frac{\rho^2}{2t}}\sum_{j \in \mathbb{Z}} \left[{\it I}_{|(2j+1)|p -(1/2)}\left(\frac{\rho^2}{4t}\right) + {\it I}_{|(2j+1)|p +(1/2)}\left(\frac{\rho^2}{4t}\right)\right] \frac{\sin[2(2j+1)p\phi]}{2j+1},
\end{equation*}
which is, up to the normalizing factor $1/\sqrt{\pi}$, formula (2.20) in \cite{Bra} for dihedral wedges $\beta = \pi/(2p), p \geq 1$. A particular feature of this formula is its connection with a result due to Spitzer on the argument of the planar Brownian motion which allows to get the following representation: 
\begin{pro}
 Let $W_p$ be the square wave function with period $\pi/p$ and amplitude one: 
\begin{equation*}
W_p(x) := \textrm{sgn}(\sin(2px)).
\end{equation*}
and let $Z$ be a planar Brownian motion with angular process $\Theta$. Then 
\begin{equation*}
\mathbb{P}_{\rho,\phi}^{(-1/2,-1/2)}(T_0 > t) = \frac{1}{2} \mathbb{E}_{\rho,0}\left[W_p(\Theta_t+\phi)\right] = \frac{1}{2} \mathbb{E}_{\rho,\phi}\left[W_p(\Theta_t)\right],
\end{equation*}
where $\mathbb{E}_{\rho,\phi}$ stands for the expectation with respect to the probability law $\mathbb{P}_{\rho,\phi}$ of $Z$ starting at $(\rho, \phi) \in C$. 
\end{pro}
\begin{proof}:
Spitzer's result alluded to above gives an explicit expression of the angular process $\Theta_t$ at time $t > 0$ (\cite{Spi}, eq. 2.10. p.193): 
\begin{equation}\label{Spi}
\mathbb{E}_{\rho,0}(e^{i\lambda \Theta_t}) = \frac{\sqrt{\pi}}{2} \sqrt{\frac{\rho^2}{2t}}e^{-\rho^2/(4t)} \left[{\it I}_{(|\lambda| -1)/2}\left(\frac{\rho^2}{4t}\right) + {\it I}_{(|\lambda| +1)/2}\left(\frac{\rho^2}{4t}\right)\right], \quad \lambda \in \mathbb{R}.
\end{equation}
It follows that 
\begin{align*}
\mathbb{P}_{\rho,\phi}^{(-1/2,1/2)}(T_0 > t) & = \frac{2}{\pi}\sum_{j \geq 0} \mathbb{E}_{\rho,0}\left[e^{2i(2j+1)p\Theta_t}\right] \frac{\sin[2(2j+1)p\phi]}{2j+1} 
\\& = \frac{1}{\pi} \sum_{j \in \mathbb{Z}} \mathbb{E}_{\rho,0}\left[e^{2i(2j+1)p\Theta_t}\right] \frac{\sin[2(2j+1)p\phi]}{2j+1}
\\ & = \frac{1}{\pi} \mathbb{E}_{\rho,0}\left\{ \sum_{j \in \mathbb{Z}}  \cos[2(2j+1)p\Theta_t]\frac{\sin[2(2j+1)p\phi]}{2j+1}\right\}
\\& = \frac{1}{2\pi}  \mathbb{E}_{\rho,0}\left\{\sum_{j \in \mathbb{Z}} \frac{1}{2j+1}\Bigl[\sin\bigl[2(2j+1)p(\Theta_t+\phi)\bigr] - \sin\bigl[2(2j+1)p(\Theta_t -\phi)\bigr]\Bigr]\right\}
\\& = \frac{1}{\pi} \mathbb{E}_{\rho,0}\left\{\sum_{j \in \mathbb{Z}} \frac{1}{2j+1}\sin[2(2j+1)p(\Theta_t+\phi)] \right\}
\end{align*}
where the second equality follows from the fact that $\Theta_t$ and $-\Theta_t$ have the same distribution. Now, recall the Fourier series of $W_p$: 
\begin{equation*}
\hat{W_p}(x) = \frac{2}{\pi}\sum_{j \in \mathbb{Z}} \frac{\sin(2(2j+1)px)}{2j+1}
\end{equation*}
which converges pointwise to $W_p$ outside the countable set of discontinuity points of the latter. Since the distribution of $\Theta_t$ has no atoms, then
\begin{equation*}
W_p(\Theta_t+\phi) = \frac{2}{\pi} \sum_{j \in \mathbb{Z}} \frac{1}{2j+1}\sin[2(2j+1)p(\Theta_t+\phi)]  
\end{equation*}
almost surely, whence the proposition follows.
\end{proof}

Note that 
\begin{align*}
\mathbb{E}_{\rho,\phi}\left[W_p(\Theta_t)\right] &= \mathbb{E}_{\rho,\phi}\left[W_p(\Theta_t){\bf 1}_{\{T_0 > t\}}\right]  + \mathbb{E}_{\rho,\phi}\left[W_p(\Theta_t){\bf 1}_{\{T_0 < t\}}\right] 
\\& = \mathbb{P}_{\rho,\phi}(T_0 > t) + \mathbb{E}_{\rho,\phi}\left[W_p(\Theta_t){\bf 1}_{\{T_0 < t\}}\right]
\end{align*}
since $W_p(x) = 1$ for $x \in (0,\pi/(2p))$. Consquently, the last proposition asserts that 
\begin{align*}
\mathbb{P}_{\rho,\phi}(T_0 > t) = \mathbb{E}_{\rho,\phi}\left[W_p(\Theta_t){\bf 1}_{\{T_0 < t\}}\right]. 
\end{align*}

{\bf Acknowledgments}: The author would like to thank A. Comtet for providing him with a copy of his unpublished manuscript where the tail distribution of the first exit time from an arbitrary wedge by a planar Brownian motion is derived. Thanks also to L. Del\'eaval for his careful reading of the paper.


\begin{thebibliography}{99}
\bibitem{AAR}\emph{G. E. Andrews, R. Askey, R. Roy}. Special functions. {\it Cambridge University Press}. 1999.
\bibitem{BBO}\emph{P. Biane, P. Bougerol, N. O' Connell.} Continuous crystal and Duistermaat-Heckman measure for Coxeter groups. \textit{Adv. Maths},  \textbf{221}, (2009), 1522-1583. 
\bibitem{Bra}\emph{S. Brassesco}. A note on planar Brownian motion. {\it Ann. Probab.} {\bf 20}, 1992, no.3, 1498-1503. 
\bibitem{CDGRVY} \emph{O. Chybiryakov, N. Demni, L. Gallardo, M. R\"osler, M. Voit, M. Yor}. Harmonic and Stochastic Analysis of Dunkl Processes. {\it Travaux en Cours, Hermann}. 2008. 
\bibitem{Com}\emph{A. Comtet}. Private communication. 
\bibitem{Demni0}\emph{N. Demni}. Radial Dunkl processes associated with dihedral systems. {\it Sem. Probab.} {\bf XLII}. 153-169, Lecture Notes in Math., 1979, Springer, Berlin, (2009). 
\bibitem{Demni1}\emph{N. Demni}. Radon Transform on spheres and generalized Bessel function associated with dihedral groups. {\it J. Lie Theory}, {\bf 22} (2012), no. 1, 81-91.
\bibitem{DL}\emph{N. Demni, D. L\'epingle}. Brownian Motion, Reflection Groups and Tanaka Formula. {\it Rendi. Sem. Mat. Univ. Padova}, {\bf 127}, (2012). 41-55. 
\bibitem{Dou-Oco}\emph{Y. Doumerc Y., N. O'Connell}. Exit problems associated with finite reflection groups, {\it Probab. Theory Related Fields}  {\bf 132} (2005), 501--538.
\bibitem{Dun-Xu}\emph{C. F. Dunkl, Y. Xu}. Orthogonal Polynomials of Several Variables. {\it Encyclopedia of Mathematics and Its Applications. Cambridge University Press}. (2001).
\bibitem{Er} \emph{A. Erdelyi, W. Magnus, F. Oberhettinger, F. G. Tricomi}. Higher transcendental functions. {\bf Vol. I.} {\it McGraw-Hill Book Company, Inc., New York-Toronto-London}, (1953). xvii+396 pp.  
\bibitem{Erd} \emph{A. Erdelyi, W. Magnus, F. Oberhettinger, F. G. Tricomi}. Higher transcendental functions. {\bf Vol. II.} {\it McGraw-Hill Book Company, Inc., New York-Toronto-London}, (1953). xvii+396 pp.  
\bibitem{Hum}\emph{J. E. Humphreys}. Reflections Groups and Coxeter Groups. {\it Cambridge University Press}. {\bf 29}. (2000).
\bibitem{Kat-Tan}\emph{M. Katori, H. Tanemura.} Symmetry of matrix-valued stochastic processes and noncolliding diffusion particle systems. {\it J. Math. Phys}. {\bf 45} (2004), no. 8, 3058-3085.
\bibitem{MY}\emph{H. Matsumoto, M. Yor}. Exponential functionals of Brownian motion, I: probability laws at fixed time. {\it Probab. Surv}. {\bf Vol 2}, (2005), 312-347. 
\bibitem{Rev-Yor}\emph{D. Revuz, M. Yor}. Brownian Motion and Continuous Martingales. {\it Third Edition. Springer}. (1999). 
\bibitem{Spi}\emph{F. Spitzer}. Some theorems concerning $2$-dimensional Brownian motion. {\it Trans. A.M.S}. {\bf 87}, (1958), no.1, 187-197. 
\bibitem{Xu}\emph{Y. Xu}. A product formula for Jacobi polynomials. {\it Special functions (Hong Kong, 1999), 423-430, World Sci. Publ., River Edge, NJ}, (2000). 
\bibitem{Vak-Yor}\emph{S. Vakeroudis, M. Yor}. Integrability properties and limit theorems for the exit time from a cone of planar Brownian motion. {\it Bernoulli}. {\bf 19}, (2013). 2000-2009. 
\end{thebibliography}
\end{document}